\documentclass{article}
\usepackage[total={6in, 9in}]{geometry}
\usepackage[utf8]{inputenc}
\usepackage[]{graphicx}
\usepackage{amsthm,amsmath,amssymb,tikz,tikz-cd,amsmath,mathrsfs,mathtools,multicol,dirtytalk,tabularx,xy,lipsum,url,enumitem,cmll}
\usepackage{float}
\usepackage{enumitem}

\usepackage{blkarray}
\usepackage{tabularray}
\usepackage[colorlinks=true,linkcolor=blue,citecolor=blue]{hyperref}

\numberwithin{equation}{section}

\theoremstyle{plain}
\newtheorem{theorem}{Theorem}[section]
\newtheorem{lemma}[theorem]{Lemma}
\newtheorem{proposition}[theorem]{Proposition}

\theoremstyle{definition}
\newtheorem{definition}[theorem]{Definition}
\newtheorem{remark}[theorem]{Remark}

\setlist[itemize]{noitemsep}

\title{\textbf{Orthogonality relations and operators on bounded quasi-implication algebras}}

\author{Joseph McDonald\footnote{University of Alberta, Department of Philosophy, Edmonton, Canada, T6G 2E7 \\Email: jsmcdon1@ualberta.ca}}
\date{}
\begin{document}

\maketitle
\begin{abstract}
   In this note, we study various relational and algebraic aspects of the bounded quasi-implication algebras introduced by Hardegree \cite{hardegree1, hardegree2}. By generalizing the constructions given by MacLaren \cite{maclaren} and Goldblatt \cite{goldblatt} within the setting of ortholattices, we construct various orthogonality relations from bounded quasi-implication algebras.

   We then introduce certain bounded quasi-implication algebras with an additional operator, which we call \emph{monadic quasi-implication algebras}, and study them within the setting of quantum monadic algebras. A \emph{quantum monadic algebra} is an orthomodular lattice equipped with a closure operator, known as a \emph{quantifier}, whose closed elements form an orthomodular sub-lattice. It is shown that every quantum monadic algebra can be converted into a monadic quasi-implication algebra with the underlying magma structure being determined by the operation of Sasaki implication on the underlying orthomodular lattice. It is then conversely demonstrated that every monadic quasi-implication algebra can be converted into a quantum monadic algebra. These constructions are shown to induce an isomorphism between the category $\mathbf{QMA}$ of quantum monadic algebras and the category $\mathbf{MQIA}$ of monadic quasi-implication algebras.

   Finally, by generalizing the constructions given by Harding \cite{harding} as well as Harding, McDonald, and Peinado \cite{harding3} in the setting of monadic ortholattices, we construct various monadic orthoframes from monadic quasi-implication algebras.    
\par
\vspace{.2cm}
\end{abstract}

\section{Introduction}\label{sec:intro}
A \emph{Boolean implication algebra} is a magma $\langle A;\cdot\rangle$ satisfying the contraction, quasi-commutativity, and exchange equations:\footnote{We use the term \emph{magma} in the sense of universal algebra to refer to a set $A$ equipped with a binary operation $\cdot\colon A^2\to A$ which is closed with respect to $A$. Magma are also commonly referred to as \emph{binars} (see Bergman \cite{bergman}.)}     
    \begin{enumerate}
        \item $(x\cdot y)\cdot x=x$;
        \item $(x\cdot y)\cdot y=(y\cdot x)\cdot x$; 
        \item $x\cdot (y\cdot z)=y\cdot(x\cdot z)$.
    \end{enumerate}
Boolean implication algebras were introduced by Abbott in \cite{abbott671,abbott672} for the purposes of studying the algebraic properties of the operation of implication within the classical propositional calculus. Abbott showed that every Boolean implication algebra admits of an underlying partial ordering defined by $x\preceq y\Longleftrightarrow x\cdot y=1$ as well as an upper bound constant defined by $1:=x\cdot x$. Moreover, by equipping any Boolean implication algebra $A$ with a lower bound constant $0$ defined by $0\cdot x=1$, the algebraic structure $\langle A;\odot,\oplus,^*,0,1\rangle$ forms a Boolean algebra where $x^*:=x\cdot 0$ is the operation of Boolean complementation, $x\oplus y:=(x\cdot y)\cdot y$ is the operation of least upper bound, and $x\odot y:=(x^*\oplus y^*)^*$ is the operation of greatest lower bound. Conversely, Abbott demonstrated that by defining an operation $p_c\colon B\times B\to B$ on a Boolean algebra $B$ by the lattice polynomial $p_c(x,y):=x^{\perp}\vee y$, the magma $\langle B;p_c\rangle$ becomes a bounded Boolean implication algebra. Closely related to the Boolean implication algebras are the Hilbert algebras studied in \cite{diego} as well as the $L$-algebras introduced in \cite{rump}. Hardegree \cite{hardegree1, hardegree2} generalized implication algebras by introducing quasi-implication algebras which model the operation of Sasaki implication within orthomodular quantum logic -- the logical calculus corresponding to orthomodular lattices. Just as in the case of bounded Boolean implication algebras and Boolean algebras, the theory of bounded quasi-implication algebras stands in a 1-1 correspondence with the theory of orthomodular lattices.

The first part of this paper investigates some relational aspects of bounded quasi-implication algebras. In particular, by mimicking the constructions given by MacLaren \cite{maclaren} and Goldblatt \cite{goldblatt} in the setting of ortholattices, we construct certain orthogonality relations from bounded quasi-implication algebras through their non-zero elements as well as through their proper filters. In the second part of this paper, monadic quasi-implication algebras are introduced as certain bounded quasi-implication algebras equipped with an additional operator. It is first shown that every monadic quasi-implication algebra gives rise to a quantum monadic algebra. A quantum monadic algebra is an orthomodular lattice equipped with a closure operator, known as a quantifier, whose closed elements form an orthomodular sub-lattice. Quantum monadic algebras were first studied by Janowitz \cite{janowitz} and later by Harding \cite{harding} and form non-distributive generalizations of the monadic Boolean algebras studied by Halmos \cite{halmos}, which provide an algebraic model for classical predicate calculus in a single variable. It is then conversely demonstrated that every quantum monadic algebra can be converted into a monadic quasi-implication algebra with the underlying magma structure being given by the operation of Sasaki implication on the underlying orthomodular lattice. These constructions are shown to induce an isomorphism between the category $\mathbf{QMA}$ of quantum monadic algebras and the category $\mathbf{MQIA}$ of monadic quasi-implication algebras. Finally, by generalizing the constructions given by Harding, McDonald, and Peinado in the setting of monadic ortholattices, we construct various monadic orthoframes from a monadic quasi-implication algebra. 

\section{Bounded quasi-implication algebras}
We first provide some preliminaries for orthomodular lattices and bounded quasi-implication algebras. For more details on the former, consult \cite{kalmbach} and for the latter, consult \cite{hardegree1}. We will often conflate between an algebra $\langle A;(\odot^n_{i})_{i\in I}\rangle$ and its underlying carrier set $A$.  

\begin{definition}\label{ortholattice}
    An \emph{ortholattice} is an algebra $\langle A;\wedge,\vee,^{\perp},0,1\rangle$ of similarity type $\langle 2,2,1,0,0\rangle$ satisfying the following conditions: 
    \begin{enumerate}
        \item $\langle A;\wedge,\vee,0,1\rangle$ is a bounded lattice;
        \item the operation $^{\perp}\colon A\to A$ is an \emph{orthocomplementation}: 
        \begin{enumerate}
             \item $x\wedge x^{\perp}=0$; $x\vee x^{\perp}=1$;
            \item $x\leq y\Rightarrow y^{\perp}\leq x^{\perp}$;
            \item $x=x^{\perp\perp}$.
        \end{enumerate}
    \end{enumerate}
\end{definition}

Note that conditions $2(a)-2(c)$ are the requirement that the operation of orthocomplementation be a complementation that is an order-inverting involution. Ortholattices can be defined by a finite set of equations and hence form a variety in the sense of universal algebra (see for instance \cite{bergman}). Ortholattices, unlike Boolean algebras, are not necessarily distributive. Indeed, an ortholattice $A$ is a Boolean algebra if and only if $A$ is distributive.    
\begin{definition}\label{oml}
    An \emph{orthomodular lattice} is an ortholattice $\langle A;\wedge,\vee,^{\perp},0,1\rangle$ satisfying the following quasi-equation:
    \[x\leq y\Longrightarrow y=x\vee(x^{\perp}\wedge y).\]
\end{definition}

Prototypical examples of orthomodular lattices include the lattice $\mathcal{C}(H)$ of closed linear subspaces of a Hilbert space $H$. The following gives several well-known characterizations of orthomodular lattices.

\begin{proposition}[\cite{kalmbach}]\label{oml characterization}
For any ortholattice $A$, the following are equivalent: 
\begin{enumerate}
    \item $A$ is an orthomodular lattice (in the sense of Definition \ref{oml});
    \item $A$ satisfies the condition that if $x\leq y$ and $x^{\perp}\wedge y=0$, then $x=y$;
    \item $A$ satisfies  $x\vee y=x\vee(x^{\perp}\wedge(x\vee y))$ or $x\wedge y=x\wedge(x^{\perp}\vee(x\wedge y))$;
    \item $A$ contains no sublattice isomorphic to the Benzene ring: 
    
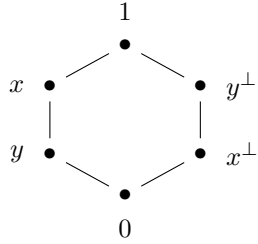
\[
 \begin{tikzpicture}
  \node (a) at (0,1) [label=above:$1$] {$\bullet$};
  \node (d) at (0,-1) [label=below:$0$] {$\bullet$};
  \node (e) at (-1,-.45) [label=left:$y$] {$\bullet$};
  \node (f) at (-1,.45) [label=left:$x$] {$\bullet$};
    \node (g) at (1,-.45) [label=right:$x^{\perp}$] {$\bullet$};
  \node (h) at (1,.45) [label=right:$y^{\perp}$] {$\bullet$};
  \draw (e) -- (f) (g) -- (h) (e) -- (d) (g) -- (d) (f) -- (a) (h) -- (a) (a);
  \draw[preaction={draw=white, -,line width=6pt}];
  \end{tikzpicture}
  \]
\end{enumerate}
\end{proposition}

Hardegree \cite{hardegree1} described the following as the minimal conditions that an implication $\cdot\colon A\times A\to A$ should be required to satisfy in an ortholattice:
\begin{enumerate}
    \item $x\leq y\Rightarrow x\cdot y=1$;
    \item $x\wedge(x\cdot y)\leq y$;
    \item $y^{\perp}\wedge(x\cdot y)\leq x^{\perp}$.
\end{enumerate}
These conditions form the algebraic counterparts to the laws of implication, modus ponens, and modus tollens, respectively. Hardegree \cite{hardegree3} showed that there are exactly 3 ortholattice polynomials satisfying these conditions in an orthomodular lattice: 
\begin{enumerate}
    \item $p_s(x,y)=x^{\perp}\vee(x\wedge y)$ (Sasaki implication);
    \item $p_d(x,y)=(x^{\perp}\wedge y^{\perp})\vee y$ (Dishkant implication);
    \item $p_k(x,y)=(x\wedge y)\vee(x^{\perp}\wedge y)\vee(x^{\perp}\wedge y^{\perp})$ (Kalmbach implication).
\end{enumerate}
One can easily demonstrate that when an ortholattice is distributive (i.e., a Boolean algebra), then: \[p_c(x,y)=p_s(x,y)=p_d(x,y)=p_k(x,y).\] Notice that in the Benzene ring depicted in Proposition \ref{oml characterization} that:
\[x\wedge p_c(x,y)=x\wedge(x^{\perp}\vee y)=x\wedge 1=x>y\]
so $p_c$ (i.e., classical implication) does not in general satisfy the algebraic implicational law of modus ponens in an ortholattice.

We now describe the variety of quasi-implication algebras introduced by Hardegree \cite{hardegree1, hardegree2}, which serve as an algebraic model for the operation of Sasaki implication on an orthomodular lattice. 
\begin{definition}\label{quasi-implication algebra}
    A \emph{quasi-implication algebra} is a magma $\langle A;\cdot\rangle$ satisfying: 
    \begin{enumerate}
        \item $(x\cdot y)\cdot x=x$;
        \item $(x\cdot y)\cdot(x\cdot z)=(y\cdot x)\cdot(y\cdot z)$;
        \item $((x\cdot y)\cdot(y\cdot x))\cdot x=((y\cdot x)\cdot(x\cdot y))\cdot y$.
    \end{enumerate}
\end{definition}
The following is an easy consequence of the axioms for quasi-implication algebras given above.  
\par
\vspace{1cm}
\begin{lemma}[\cite{hardegree1}]\label{lemma}
    Any quasi-implication algebra $\langle A;\cdot\rangle$ satisfies:
    \begin{enumerate}
        \item $x\cdot(x\cdot y)=x\cdot y$;
        \item $x\cdot x=(x\cdot y)\cdot(x\cdot y)$;
        \item $x\cdot x=y\cdot y$.
    \end{enumerate}
    \end{lemma}
\begin{definition}\label{top element}
    Let $A$ be a quasi-implication algebra. Then define a constant $1\in A$ by $1:=x\cdot x$. 
\end{definition}
Note that by virtue of condition 3 of Lemma \ref{lemma}, the constant element $1$ is well-defined in any quasi-implication algebra. 
\begin{proposition}[\cite{hardegree1}]\label{top element lemma}
    Every quasi-implication algebra $A$ satisfies $1\cdot x=x$ and $x\cdot 1=1$ for all $x\in A$. 
\end{proposition}

The following shows that every Boolean implication algebra is a quasi-implication algebra and that the equational theory of quasi-implication algebras is satisfied by the operation of Sasaki implication in an orthomodular lattice. 
 
\begin{lemma}[\cite{hardegree1}]\label{implication is quasi-imlpication}
   Every Boolean implication algebra is a quasi-implication algebra. Moreover, if $A$ is an orthomodular lattice, then $\langle A;p_s\rangle$ is a quasi-implication algebra.   
\end{lemma}

Although the following was not explicitly stated by Hardegree in \cite{hardegree1}, we show that quasi-implication algebras are proper generalizations of Boolean implication algebras. 

\begin{proposition}
    There exist quasi-implication algebras that are not Boolean implication algebras. 
\end{proposition}

\begin{proof}
 To construct an example of a quasi-implication algebra that is not an implication algebra, consider the orthomodular lattice: 
    \par
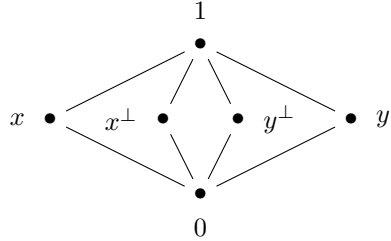
\vspace{.3cm}
    \begin{center}
    \begin{tikzpicture}
  \node (a) at (0,1) [label=above:$1$] {$\bullet$};
  \node (b) at (2,0) [label=right:$y$] {$\bullet$};
  \node (c) at (-2,0) [label=left:$x$] {$\bullet$};
  \node (d) at (0,-1) [label=below:$0$] {$\bullet$};
  \node (i) at (.5,0) [label=right:$y^{\perp}$] {$\bullet$};
  \node (j) at (-.5,0) [label=left:$x^{\perp}$] {$\bullet$};
  \draw  (i) -- (a) (j) -- (a) (i) -- (d) (j) -- (d) (d) -- (b) (d) -- (c) (b) -- (a) (c) -- (a);
  \draw[preaction={draw=white, -,line width=6pt}];
\end{tikzpicture}
\end{center}

Its corresponding quasi-implication algebra is given by the table below: 

\par
\vspace{.3cm}

\begin{table}[htbp]
\centering
\begin{tabular}{|c|c|c|c|c|c|c|c|c|c|c|c|c|c|c|c|c|}
\hline
$p_s$ & $x$ & $x^{\perp}$ & $y$ & $y^{\perp}$ & $0$ & $1$ \\
\hline
$x$ & 1 & $x^{\perp}$ & $x^{\perp}$ & $x^{\perp}$ & $x^{\perp}$ & 1 \\
\hline
$x^{\perp}$ & $x$ & 1 & $x$ & $x$ & $x$ & 1 \\
\hline
$y$ & $y^{\perp}$ & $y^{\perp}$ & 1 & $y^{\perp}$ & $y^{\perp}$ & 1 \\
\hline
$y^{\perp}$ & $y$ & $y$ & $y$ & 1 & $y$ & 1 \\
\hline
0 & 1 & 1 & 1 & 1 & 1 & 1 \\
\hline
1 & $x$ & $x^{\perp}$ & $y$ & $y^{\perp}$ & 0 & 1 \\
\hline
\end{tabular}
\end{table}
This is indeed a quasi-implication algebra by part 2 of Lemma \ref{implication is quasi-imlpication}. To see that it is not a Boolean implication algebra, note that $p_s(p_s(x,y),y)=p_s(x^{\perp},y)=x$ but $p_s(p_s(y,x),x)=p_s(y^{\perp},x)=y$. Therefore we find that: \[p_s(p_s(x,y),y)\not=p_s(p_s(y,x),x)\] and hence the quasi-commutativity equation that is characteristic of Boolean implication algebras is not satisfied. 
\end{proof}

\begin{lemma}[\cite{hardegree1}]\label{partial order}
If $A$ is a quasi-implication algebra, then $\preceq\subseteq A^2$ defined by $x\preceq y$ iff $x\cdot y=1$ is a partial order.      
\end{lemma}
Notice that for any quasi-implication algebra $A$, we have $x\cdot 1=1$ for all $x\in A$ by Proposition \ref{top element lemma}, and hence $x\preceq 1$ for all $x\in A$ by Lemma \ref{partial order}. Hence the constant element $1\in A$ is the greatest element in the poset $\langle A;\preceq\rangle$.     

\begin{definition}\label{bottom element}
    A \emph{bounded quasi-implication algebra} is a quasi-implication algebra $A$ equipped with a constant $0\in x$ defined by $0\cdot x=1$ for all $x\in A$. 
\end{definition}
By Lemma \ref{partial order}, it is clear that $0$ is the least element in the poset $\langle A;\preceq\rangle$ for any bounded quasi-implication algebra $A$. Moreover, since for any orthomodular lattice $A$ and any $x\in A$: \[p_s(0,x)=0^{\perp}\vee(0\wedge x)=1\vee(0\wedge x)=1\vee 0=1\]  it is obvious that $\langle A;p_s\rangle$ further induces a bounded quasi-implication algebra.

Conversely to that of Lemma \ref{implication is quasi-imlpication}, the following result shows that every bounded quasi-implication algebra induces an orthomodular lattice. 

\begin{lemma}[\cite{hardegree1}]\label{quasi-implication algebra is an orthomodular lattice}
    If $A$ is a bounded quasi-implication algebra, $\langle A;\oplus,\odot,^*,0,1\rangle$ forms an orthomodular lattice where $x^*:=x\cdot 0$ is the operation of orthogonal complementation, $x\oplus y:=((x\cdot y)\cdot(y\cdot x))\cdot x$ is the operation of least upper bound, and $x\odot y:=(x^*\oplus y^*)^*$ is the operation of greatest lower bound. 
\end{lemma}

The proceeding lemma will be exploited within the constructions and proofs throughout the remainder of this paper. 
\begin{lemma}[\cite{hardegree1}]\label{meets}
    If $A$ is a bounded quasi-implication algebra, then: \[(x^*\oplus y^*)^*=((x\cdot y)\cdot(x\cdot 0))\cdot 0\] so that $x\odot y=((x\cdot y)\cdot(x\cdot 0))\cdot 0$. 
\end{lemma}
\section{Orthogonality relations arising from bounded quasi-implication algebras}
In this section, we investigate the basic theory of orthogonality relations arising from quasi-implication algebras. 
\begin{definition}
    If $X$ is a set, an \emph{orthogonality relation} on $X$ is a binary relation $\perp$ that is irreflexive and symmetric. 
\end{definition}
We call $\langle X;\perp\rangle$ an \emph{orthoframe} whenever $\perp$ is an orthogonality relation on $X$. Orthogonality relations are special examples of the polarities described by Birkhoff \cite{birkhoff2}.

For any orthoframe $X$ and any subset $U\subseteq X$, let:  \[U^{\perp}=\{x\in X:x\perp U\}=\{x\in X:x\perp y\hspace{.1cm}\text{for all}\hspace{.1cm}y\in U\}\] For any orthoframe $X$, we call $U\subseteq X$ \emph{bi-orthogonally closed} if $U=U^{\perp\perp}$ and denote the set of all bi-orthogonally closed subsets of $X$ by $\mathcal{B}(X)$, i.e., \[\mathcal{B}(X)=\{U\in\wp(X):U=U^{\perp\perp}\}.\] 

Orthogonality relations play an important role in the representation theory of general ortholattices. 

\begin{theorem}[\protect{\cite{birkhoff2}}]\label{complete ol}
    If $X$ is an orthoframe, then $\langle\mathcal{B}(X);\cap,\sqcup,^{\perp},\emptyset,X\rangle$ is a complete ortholattice where $U\sqcup V:=(U\cup V)^{\perp\perp}$ for all $U,V\in\mathcal{B}(X)$. 
\end{theorem} 
For an ortholattice $A$, MacLaren \cite{maclaren} constructed an orthogonality relation $\perp\subseteq A\setminus\{0\}\times A\setminus\{0\}$ by defining $x\perp y\Leftrightarrow x\leq y^{\perp}$. MacLaren demonstrated that the induced complete ortholattice of bi-orthogonally closed subsets of $\langle A\setminus\{0\};\perp\rangle$ gives rise to the MacNeille completion of $A$ where the associated meet-dense and join-dense embedding is given by $\psi(x)=\{y\in A\setminus\{0\}:y\leq x\}$.

The definition of the partial-order structure $\preceq$ induced by a bounded quasi-implication algebra $A$ as well as the definition of complementation in $A$ suggests the following generalization of the above construction of an orthogonality relation to the setting of bounded quasi-implication algebras.  

\begin{definition}
    Let $A$ be a bounded quasi-implication algebra and define a binary relation $\perp^M_A$ on $A\setminus\{0\}$ by $x\perp^M_Ay\Leftrightarrow x\cdot (y\cdot 0)=1$.  
\end{definition}
For any bounded quasi-implication algebra $A$, we call $X^M_A=\langle A\setminus\{0\};\perp^M_A\rangle$ the \emph{MacLaren frame} of $A$.

\begin{proposition}
     If $A$ is a bounded quasi-implication algebra, then its MacLaren frame $X_A^M$ is an orthoframe.
\end{proposition}
\begin{proof}
    We first demonstrate that $\perp_A^M$ forms an irreflexive relation on $A\setminus\{0\}$. Hence, assume for the sake of contradiction that there exists $x\in A\setminus\{0\}$ such that $x\perp_{A}x$ so that $x\cdot(x\cdot 0)=1$. However $x\cdot(x\cdot 0)=x\cdot 0$ by Lemma \ref{lemma}(1), so $x\cdot 0=1$ and hence $x=0$, which contradicts our hypothesis that $x\in A\setminus\{0\}$.

   For symmetry, assume $x\perp_{A}y$ for $x,y\in A\setminus\{0\}$, so that $x\cdot(y\cdot 0)=1$. Notice that we have $0\cdot(((y\cdot 0)\cdot x)\cdot 0)=1$ since $0\cdot z=1$ for all $z\in A$. Therefore by \cite[Lemma 10]{hardegree1} we have the following: \[((y\cdot 0)\cdot 0)\cdot((y\cdot 0)\cdot(((y\cdot 0)\cdot x)\cdot 0))=1.\] By \cite[Lemma 6 and Q2]{hardegree1}, our hypothesis that $x\perp^M_Ay$, and Proposition \ref{top element lemma}, we have:
    \begin{align*}
        (y\cdot 0)\cdot(((y\cdot 0)\cdot x)\cdot 0)&=((y\cdot 0)\cdot x)\cdot((y\cdot 0)\cdot 0)\\&=(x\cdot(y\cdot 0))\cdot(x\cdot 0)\\&=1\cdot (x\cdot 0)=x\cdot 0
    \end{align*}
    Therefore $((y\cdot 0)\cdot 0)\cdot(x\cdot 0)=1$ but $((y\cdot 0)\cdot 0)=y$ by \cite[Lemma 26]{hardegree1} so $y\cdot(x\cdot 0)=1$ and thus we conclude $y\perp^M_{A}x$. Thus $X_A^M$ forms an orthoframe. 
\end{proof}

Recall that if $A$ is a bounded lattice, a non-empty subset $\alpha\subseteq A$ is a \emph{filter} if $\alpha$ is upwards closed and closed under finite meets. A filter $\alpha\subseteq A$ is \emph{proper} if $0\not\in \alpha$, i.e., $\alpha\not=A$. We denote by $\mathfrak{F}(A)$ the proper filters of $A$.  

For an ortholattice $A$, Goldblatt \cite{goldblatt} constructed an orthogonality relation $\perp\subseteq\mathfrak{F}(A)\times\mathfrak{F}(A)$ by defining $\alpha\perp\beta$ if and only if there exists $x\in A$ such that $x\in\alpha$ and $x^{\perp}\in\beta$. Harding, McDonald, and Peinado \cite{harding3} demonstrated that the induced complete ortholattice of bi-orthogonally closed subsets of $\langle \mathfrak{F}(A);\perp\rangle$ gives rise to the canonical completion and canonical extension of $A$ where the associated dense and compact embedding is given by $\phi(x)=\{\alpha\in\mathfrak{F}(A):x\in\alpha\}$.

The definition of the partial-order structure $\preceq$ induced by a bounded quasi-implication algebra $A$, the definition of meets in $A$, and Lemma \ref{meets} suggests the following definition of a proper filter in a bounded quasi-implication algebra. 
\begin{definition}
    Let $A$ be a bounded quasi-implication algebra. A \emph{filter} on $A$ is a non-empty subset $\alpha\subseteq A$ satisfying the following conditions: 
    \begin{enumerate}
        \item if $x\in\alpha$ and $x\cdot y=1$, then $y\in\alpha$; 
        \item if $x\in \alpha$ and $y\in \alpha$, then $((x\cdot y)\cdot(x\cdot 0))\cdot 0\in \alpha$
    \end{enumerate}
    We regard a filter $\alpha\subseteq A$ as \emph{proper} whenever $0\not\in \alpha$.  
\end{definition}

For any bounded quasi-implication algebra $A$, let $\mathfrak{F}(A)$ denote the set of all proper filters of $A$. The following definition mirrors the construction of an orthogonality relation via the proper filters of an ortholattice given by Goldblatt. 
\begin{definition}
    Let $A$ be a bounded quasi-implication algebra. Define $\perp^G_A\subseteq\mathfrak{F}(A)\times\mathfrak{F}(A)$ by $\alpha\perp^G_A \beta$ iff there exists $x\in A$ such that $x\in \alpha$ and $x\cdot 0\in \beta$. 
\end{definition}
For any bounded quasi-implication algebra $A$, we call $\langle\mathfrak{F}(A);\perp^G_A\rangle$ the \emph{Goldblatt frame} of $A$.   
\begin{proposition}
    If $A$ is a bounded quasi-implication algebra, then its Goldblatt frame $X_A^G$ is an orthoframe. 
\end{proposition}
\begin{proof}
   
 For irreflexivity, assume for the sake of contradiction that $\alpha\perp_A\alpha$ for some $\alpha\in\mathfrak{F}(A)$. Then there exists $x\in A$ such that $x\in \alpha$ and $x\cdot0\in\alpha$. Since $\alpha$ is a filter, we have $((x\cdot(x\cdot 0))\cdot(x\cdot 0))\cdot 0\in \alpha$, but by Lemma \ref{lemma}(1), Definition \ref{top element}, and Proposition \ref{top element lemma}, we have: 
    \begin{align*}
        ((x\cdot(x\cdot 0))\cdot(x\cdot 0))\cdot 0=((x\cdot 0)\cdot(x\cdot 0))\cdot 0=1\cdot 0=0
    \end{align*}
and hence $0\in \alpha$, but this contradicts our hypothesis that $\alpha$ is a proper filter. 

For symmetry, assume $\alpha\perp^G_A\beta$ so that there exists $x\in A$ such that $x\in \alpha$ and $x\cdot 0\in \beta$. Note that by \cite[Lemma 26]{hardegree1}, we have $x=(x\cdot 0)\cdot 0$ and therefore $(x\cdot 0)\cdot 0\in\alpha$ with $x\cdot 0\in \beta$ which implies $\beta\perp^G_A\alpha$. This completes the proof that the Goldblatt frame $X_A^G$ of $A$ is an orthoframe.   
\end{proof}

\section{Monadic quasi-implication algebras and quantum monadic algebras}

Quantum monadic algebras were introduced by Janowitz \cite{janowitz} and then by Harding \cite{harding} and generalize the monadic Boolean algebras studied by Halmos in \cite{halmos}. 
\begin{definition}\label{monadic ortholattice}
A \emph{monadic ortholattice} is an algebra $\langle A;\wedge,\vee,^{\perp},0,1,\exists\rangle$ of type $\langle 2,2,1,0,0,1\rangle$ satisfying the following conditions:
\begin{enumerate}
    \item $\langle A;\wedge,\vee,^{\perp},0,1\rangle$ is an ortholattice;
    \item $\exists\colon A\to A$ is a unary operator, known as a \emph{quantifier}, satisfying the following conditions:
        \begin{enumerate}
        \item $\exists 0=0$
        \item $x\leq \exists x$
        \item $\exists(x\vee y)=\exists x\vee\exists y$
        \item $\exists\exists x=\exists x$
        \item $\exists(\exists x)^{\perp}=(\exists x)^{\perp}$
    \end{enumerate}
\end{enumerate}
We call $\langle A;\wedge,\vee,^{\perp},0,1,\exists\rangle$ a \emph{quantum monadic algebra} whenever $\langle A;\wedge,\vee,^{\perp},0,1\rangle$ is an orthomodular lattice. Moreover, we call $\langle A;\wedge,\vee,^{\perp},0,1,\exists\rangle$ a \emph{monadic Boolean algebra} whenever $\langle A;\wedge,\vee,^{\perp},0,1\rangle$ is a Boolean algebra.  
\end{definition}

Noting that $\exists x=x$ iff $\exists(x^{\perp})=x^{\perp}$, conditions 2(a) through 2(e) amount to asserting that a quantifier on an orthomodular lattice is a closure operator whose closed elements form an orthomodular sub-lattice. Notice that by dualizing $\exists$ by an operator $\forall\colon A\to A$ defined by $\forall x:=(\exists x^{\perp})^{\perp}$, we obtain an interior operator whose open elements form an orthomodular sub-lattice.

\begin{remark}
    It is important to observe that unlike every monadic Boolean algebra, it is not necessarily true that $\exists(x\wedge\exists y)=\exists x\wedge\exists y$ and $\forall(x\vee\forall y)=\forall x\vee\forall y$ obtain in a quantum monadic algebra. 
\end{remark}

Notice that in any monadic ortholattice $A$, we have that $\exists a\leq b$ if and only if $a\leq\forall b$ for all $a,b\in A$ and hence $\exists$ preserves arbitrary joins and $\forall$ preserves arbitrary meets, whenever they exist. 

We now introduce monadic quasi-implication algebras. 
\par
\vspace{1cm}

\begin{definition}\label{def of mqia}
    A \emph{monadic quasi-implication algebra} is an algebra $\langle A;\cdot,0,\Diamond\rangle$ of similarity type $\langle 2,0,1\rangle$ satisfying the following conditions: 
    \begin{enumerate}
        \item $\langle A;\cdot,0\rangle$ is a bounded quasi-implication algebra; 
        \item $\Diamond\colon A\to A$ is a unary operation satisfying the following conditions: 
        \begin{enumerate}
            \item $\Diamond\Diamond x\cdot\Diamond x=1$ and $x\cdot\Diamond x=1$;
            \item $\Diamond(\Diamond x\cdot 0)=\Diamond x\cdot 0$ and $\Diamond 0=0$; 
            \item $\Diamond(((x\cdot 0)\cdot(y\cdot0))\cdot x)=((\Diamond x\cdot 0)\cdot(\Diamond y\cdot0))\cdot \Diamond x$
            
        \end{enumerate}
    \end{enumerate}
\end{definition}
The following is an immediate consequence of the definition of monadic quasi-implication algebras and the induced partial order relation described in Lemma \ref{quasi-implication algebra is an orthomodular lattice}. 
\begin{proposition}\label{idempotent}
    If $A$ is a monadic quasi-implication algebra, then $\Diamond\colon A\to A$ is an idempotent operation.  
\end{proposition}
\begin{proof}
By condition 2(a) of Definition \ref{def of mqia} we have $\Diamond\Diamond x\cdot \Diamond x=1$ and $\Diamond x\cdot\Diamond\Diamond x=1$ so $\Diamond\Diamond x\preceq \Diamond x$ and $\Diamond x\preceq\Diamond\Diamond x$. By the anti-symmetry of $\preceq$, we have $\Diamond\Diamond x=\Diamond x$. 
\end{proof}
The following will be used in the proof of Theorem \ref{thm 4.5}. 
\begin{lemma}\label{useful lemma}
    Let $A$ be an orthomodular lattice. Then: 
    \[p_s(p_s(x^{\perp},y^{\perp}),x)=x\vee y.\]
\end{lemma}
\begin{proof}
    The calculation is achieved in the following manner: 
    \begin{align*}
        p_s(s_s(x^{\perp},y^{\perp}),x)&=p_s(x^{\perp\perp}\vee(x^{\perp}\wedge y^{\perp}),x)\tag{definition of $p_s$}
        \\&=p_s(x\vee(x^{\perp}\wedge y^{\perp}),x)\tag{$^{\perp}$ is an involution}
        \\&=(x\vee(x^{\perp}\wedge y^{\perp}))^{\perp}\vee((x\vee(x^{\perp}\wedge y^{\perp}))\wedge x)\tag{definition of $p_s$}
        \\&=(x^{\perp}\wedge(x^{\perp}\wedge y^{\perp})^{\perp})\vee((x\vee(x^{\perp}\wedge y^{\perp}))\wedge x)\tag{De Morgan's Laws}
        \\&=(x^{\perp}\wedge(x^{\perp\perp}\vee y^{\perp\perp}))\vee((x\vee(x^{\perp}\wedge y^{\perp}))\wedge x)\tag{De Morgan's Laws}
        \\&=(x^{\perp}\wedge(x\vee y))\vee((x\vee(x^{\perp}\wedge y^{\perp}))\wedge x)\tag{$^{\perp}$ is an involution}
        \\&=(x^{\perp}\wedge(x\vee y))\vee x\tag{Absorption Laws}
        \\&= x\vee y \tag{Proposition \ref{oml characterization}(3)}
    \end{align*}
    Notice that the final step of the above calculation makes use of the orthomodularity identity.  
\end{proof}

The following result shows that every quantum monadic algebra can be converted into a monadic quasi-implication algebra. 
\begin{theorem}\label{thm 4.5}
    If $A$ is a quantum monadic algebra, then $\langle A;p_s,0,\exists\rangle$ is a monadic quasi-implication algebra. 
\end{theorem}
\begin{proof}
    By Lemma \ref{implication is quasi-imlpication}, $\langle A;p_s,0\rangle$ is a bounded quasi-implication algebra whenever $A$ is an orthomodular lattice and hence it suffices to demonstrate that $\exists\colon A\to A$ satisfies conditions 2(a)--2(c) in Definition \ref{def of mqia}.

    The calculation for condition 2(a) runs as follows: \[p_s(\exists\exists x,\exists x)=(\exists\exists x)^{\perp}\vee(\exists\exists x\wedge\exists x)=(\exists x)^{\perp}\vee\exists x=1.\] as well as $p_s(x,\exists x)=x^{\perp}\vee(x\wedge\exists x)=x^{\perp}\vee x=1$. For condition 2(b), it is trivial that $\exists 0=0$.  To verify $\exists p_s(\exists x,0)=p_s(\exists x,0)$ we first note that:
\begin{equation}\label{equation 4.1}
    p_s(x,0)=x^{\perp}\vee(x\wedge 0)=x^{\perp}\vee 0=x^{\perp}
\end{equation}
and hence by applying Equation \ref{equation 4.1}, condition 2(e) of Definition \ref{monadic ortholattice}, followed by another application of Equation \ref{equation 4.1}, we obtain: 
\[\exists p_s(\exists x,0)=\exists(\exists x)^{\perp}=(\exists x)^{\perp}=p_s(\exists x,0).\]
    To see that condition 2(c) of Definition \ref{def of mqia} is satisfied, it suffices to demonstrate: 
    \[ \exists p_s(p_s(p_s(x,0),p_s(y,0)),x)=p_s(p_s(p_s(\exists x,0),p_s(\exists y,0)),\exists x)\]
The calculation proceeds in the following manner: 
\begin{align*}
    \exists p_s(p_s(p_s(x,0),p_s(y,0)),x)&=\exists p_s(p_s(x^{\perp},y^{\perp}),x)\tag{by Equation \ref{equation 4.1}}
    \\&=\exists(x\vee y)\tag{by Lemma \ref{useful lemma}}
    \\&=\exists x\vee\exists y\tag{by Definition \ref{monadic ortholattice}.2(c)}
    \\&=p_s(p_s((\exists x)^{\perp},(\exists y)^{\perp}),\exists x)\tag{by Lemma \ref{useful lemma}}
    \\&=p_s(p_s(p_s(\exists x,0),p_s(\exists y,0)),\exists x)\tag{by Equation \ref{equation 4.1}}
\end{align*}

    This completes the proof that $\langle A;p_s,0,\exists\rangle$ forms a monadic quasi-implication algebra whenever $A$ is a quantum monadic algebra.     
\end{proof}

Conversely to Theorem \ref{thm 4.5}, the proceeding result shows that every monadic quasi-implication algebra can be converted into a quantum monadic algebra. 
\begin{theorem}\label{thm 4.6}
    If $A$ is a monadic quasi-implication algebra, then the algebra $\langle A;\odot,\oplus,^*,0,1,\Diamond\rangle$ is a quantum monadic algebra. 
\end{theorem}
\begin{proof}
    By Lemma \ref{quasi-implication algebra is an orthomodular lattice}, the reduct $\langle A;\odot,\oplus,^*,0,1\rangle$ is an orthomodular lattice if $\langle A;\cdot,0\rangle$ is a bounded quasi-implication algebra and hence it suffices to demonstrate that $\Diamond$ is a quantifier on the orthomodular lattice induced by $A$. 

    It is trivial that $\Diamond$ is a normal operator since $\Diamond 0=0$ by condition 2(b) of Definition \ref{def of mqia}. To see that $\Diamond$ is increasing, notice that $x\cdot\Diamond x=1$ by condition 2(a) of Definition \ref{def of mqia} so that $x\preceq\Diamond x$ for all $x\in A$.

    Condition 2(c) of Definition \ref{def of mqia} along with Lemma \ref{meets} guarantees that $\Diamond$ is additive in the sense that $\Diamond(x\oplus y)=\Diamond x\oplus\Diamond y$ as can be seen by:  
    \begin{align*}
        \Diamond(x\oplus y)&=\Diamond(((x\cdot 0)\odot(y\cdot 0))\cdot 0)\tag{Definition of $\oplus$}
        \\&=\Diamond((((x\cdot 0)\cdot(y\cdot 0))\cdot((x\cdot 0)\cdot0))\cdot0)\cdot 0)\tag{Definition of $\odot$}
        \\&=\Diamond(((x\cdot 0)\cdot(y\cdot0))\cdot x)\tag{\cite[Lemma 26]{hardegree1}}
        \\&=((\Diamond x\cdot 0)\cdot(\Diamond y\cdot0))\cdot \Diamond x\tag{Definition \ref{def of mqia}, 2(c)}
        \\&=(((\Diamond x\cdot 0)\cdot(\Diamond y\cdot 0))\cdot((\Diamond x\cdot 0)\cdot0))\cdot0)\cdot 0\tag{\cite[Lemma 26]{hardegree1}}
        \\&=((\Diamond x\cdot 0)\odot(\Diamond y\cdot 0))\cdot 0\tag{Definition of $\odot$}
        \\&=\Diamond x\oplus\Diamond y\tag{Definition of $\oplus$}
    \end{align*}

    We have already established in Proposition \ref{idempotent} that $\Diamond$ is an idempotent operator and hence it remains to verify that $\Diamond(\Diamond x)^*=(\Diamond x)^*$ for all $x\in A$. The result is obtained by condition 2(b) of Definition \ref{def of mqia} in the following manner: 
    \begin{align*}
        \Diamond(\Diamond x)^*&=\Diamond(\Diamond x\cdot 0)=\Diamond x\cdot 0=(\Diamond x)^*.
    \end{align*}
    This completes the proof that $\langle A;\odot,\oplus,^*,0,1,\Diamond\rangle$ is a quantum monadic algebra whenever $A$ is a monadic quasi-implication algebra. 
\end{proof}
\begin{theorem}
    Let $\mathfrak{A}=\langle A;\wedge,\vee,^{\perp},0,1,\exists\rangle$ be a quantum monadic algebra, let $\mathfrak{B}=\langle A;p_s,0,\exists\rangle$ be the corresponding monadic quasi-implication algebra obtained from $\mathfrak{A}$, and let $\widehat{\mathfrak{A}}=\langle A;\odot,\oplus,^*,0,1,\exists\rangle$ be the corresponding quantum monadic algebra obtained from $\mathfrak{B}$. Then $\mathfrak{A}=\widehat{\mathfrak{A}}$.  
\end{theorem}
\begin{proof}
    First observe that the underlying carrier set of $\mathfrak{A}$ is identical to the underlying carrier set of $\widehat{\mathfrak{A}}$. We also have that: 
    \[x\leq y\Longleftrightarrow p_s(x,y)=1\Longleftrightarrow x\preceq y.\] This fact can be demonstrated algebraically as follows: 
    \begin{align*}
        x\wedge y&=(x^{\perp}\vee(x\wedge y))\wedge x\tag{by Proposition \ref{oml characterization}(3)}
        \\&=p_s(x,y)\wedge x\tag{by the definition of $p_s$}
        \\&=p_s(x,y)\wedge(p_s(x,y)^{\perp}\vee x)\tag{by the absorption laws}
        \\&=p_s(p_s(x,y),x^{\perp})^{\perp}\tag{definition of $p_s$, De Morgan's laws}
        \\&=p_s(p_s(p_s(x,y),p_s(x,0)),0)\tag{since $x^{\perp}=p_s(x,0)$}
        \\&=x\odot y\tag{definition of $\odot$}
    \end{align*}
    For the operation of orthocomplementation, we have already seen that $x^{\perp}=p_s(x,0)=x^{*}$ for all $x\in A$. Clearly the quanifiers associated with $\mathfrak{A}$ and $\widehat{\mathfrak{A}}$ are identical. 
\end{proof}
 \begin{theorem}
     Let $\mathfrak{A}=\langle A;\cdot,0,\Diamond\rangle$ be a monadic quasi-implication algebra, let $\mathfrak{B}=\langle A;\odot,\oplus,^*,0,1,\Diamond\rangle$ be the corresponding quantum monadic algebra obtained from $\mathfrak{A}$, and let $\widehat{\mathfrak{A}}=\langle A;p_s,0,\Diamond\rangle$ be the corresponding monadic quasi-implication algebra obtained from $\mathfrak{B}$. Then $\mathfrak{A}=\widehat{\mathfrak{A}}$. 
 \end{theorem}
\begin{proof}
    It suffices to demonstrate the following equation: 
    \[x^*\oplus(x\odot y)=x\cdot y\] for all $x,y\in A$ where $x^*$, $x\oplus y$, and $x\odot y$ are defined as in the statement of Lemma \ref{quasi-implication algebra is an orthomodular lattice} but this has already been demontsrated in \cite[Lemma 30]{hardegree1}. 
\end{proof}

If $A$ and $A'$ are quantum monadic algebras, a function $h\colon A\to A'$ is a \emph{homomorphism} provided $h$ is a bounded lattice homomorphism satisfying $h(x^{\perp})=h(x)^{\perp}$ and $h(\exists x)=\exists h(x)$ for all $x\in A$. Similarly, when $A$ and $A'$ are monadic quasi-implication algebras, a function $h\colon A\to A'$ is a \emph{homomorphism} whenever $h(x\cdot y)=h(x)\cdot h(y)$, $h(0)=0'$, and $h(\Diamond x)=\Diamond h(x)$ for all $x\in A$. Clearly every homomorphism between monadic quasi-implication algebras preserves the top element since $h(1)=h(x\cdot x)=h(x)\cdot h(x)=1'$.

\begin{definition}
By $\mathbf{QMA}$ we denote the category of quantum monadic algebras and their homomorphisms. Moreover, by $\mathbf{MQIA}$ we denote the category of monadic quasi-implication algebras and their homomorphisms. 
\end{definition}
\begin{theorem}
    $\mathbf{QMA}$ is isomorphic to $\mathbf{MQIA}$. 
\end{theorem}
\begin{proof}
    In light of the results collected in this section thus far, it suffices to demonstrate a bijective correspondence between homomorphisms of quantum monadic algebras and homomorphisms of monadic quasi-implication algebras.  First, let $A$ and $A'$ be quantum monadic algebras and let $h\colon A\to A'$ be a homomorphism. Then:  
\begin{align*}
    h(p_s(x,y))&=h(x^{\perp}\vee(x\wedge y))=h(x^{\perp})\vee h(x\wedge y)\\&=h(x)^{\perp}\vee(h(x)\wedge h(y))=p_s(h(x),h(y))
\end{align*}
Conversely, if $A$ and $A'$ are monadic quasi-implication algebras and $h\colon A\to A'$ is a homomorphism: 
\begin{align*}
    h(x\odot y)&=h(((x\cdot y)\cdot(x\cdot 0))\cdot 0)
    \\&=((h(x)\cdot h(y))\cdot(h(x)\cdot h(0)))\cdot h(0)
    \\&=((h(x)\cdot h(y))\cdot(h(x)\cdot 0'))\cdot 0'=h(x)\odot h(y)
\end{align*}
    The remaining cases going in each direction run analogously. 
\end{proof}

\section{Monadic orthoframes induced by monadic \\ quasi-implication algebras}

Let $X$ be a set and $R$ be a binary relation on $X$. Then for any $U\subseteq X$, we will write $R[U]$ to denote the relational image of $U$ under $R$, i.e., \[R[U]=\{y\in X:xRy\hspace{.1cm}\text{for some}\hspace{.1cm}x\in U\}\] so that for any $x\in X$, we have $R[\{x\}]=\{y\in X:xRy\}$. 

 The following relational structures were introduced by Harding \cite{harding} and play an important role within the representation theory of monadic ortholattices.   
\begin{definition}\label{mof}
    A \emph{monadic orthoframe} is a triple $\langle X;\perp,R\rangle$ such that: 
    \begin{enumerate}
        \item $\langle X;\perp\rangle$ is an orthoframe;
        \item $R$ is a binary relation on $X$ that is reflexive and transitive; 
        \item $R[R[\{x\}]^{\perp}]\subseteq R[\{x\}]^{\perp}$ for all $x\in X$.  
    \end{enumerate}
\end{definition}
Condition 3 in the above definition can be viewed as the requirement that $R[\{x\}]^{\perp}$ be closed under $R$. 
Note that in any monadic orthoframe $\langle X;\perp,R\rangle$, we have  $R[R[\{x\}]^{\perp}]=R[\{x\}]^{\perp}$ for all $x\in X$, by the reflexivity of $R$. 

\begin{theorem}[\cite{harding}]
      If $X$ is a monadic orthoframe, $\langle\mathcal{B}(X);\cap,\sqcup,^{\perp},\emptyset,X,\exists_R\rangle$ is a complete monadic ortholattice where $\exists_RU:=R[U]^{\perp\perp}$ for all $U\in\mathcal{B}(X)$.
\end{theorem}

For a monadic ortholattice $A$, Harding \cite{harding} constructed a monadic orthoframe from $A$ by defining a binary relation $R^M_A\subseteq A\setminus\{0\}\times A\setminus\{0\}$ on the MacLaren frame $X^M_A=\langle A\setminus\{0\};\perp^M_A\rangle$ by $xR^M_Ay\Leftrightarrow y\leq \exists x$. Harding, McDonald, and Peinado \cite{harding3} demonstrated that this construction gives rise to the MacNeille completion of $A$. The following definition generalizes the above construction of a monadic orthoframe to the setting of monadic quasi-implication algebras. 

\begin{definition}
    Let $A$ be a monadic quasi-implication algebra. Then define a binary relation $R^M_A\subseteq A\setminus\{0\}\times A\setminus\{0\}$ by $xR^M_Ay\Leftrightarrow y\cdot\Diamond x=1$. 
\end{definition}
We call $X^M_A=\langle A\setminus\{0\};\perp^M_A,R^M_A\rangle$ the \emph{monadic MacLaren frame} of $A$. The following results will be used within the proof of Theorem \ref{monadic maclaren frame}
\begin{proposition}\label{monotone}
    $\Diamond$ is monotone with respect to the lattice order induced by any bounded quasi-implication algebra $A$.  
\end{proposition}
\begin{proof}
    For any $x,y\in A$, suppose $x\preceq y$ so that $y=x\oplus y$. It suffices to show that $\Diamond x\preceq\Diamond y$, i.e., $\Diamond y=\Diamond x\oplus\Diamond y$. The calculation is then an immediate consequence of Theorem \ref{thm 4.6} since $\Diamond y=\Diamond(x\oplus y)=\Diamond x\oplus\Diamond y$. 
    Therefore we conclude that $\Diamond x\preceq\Diamond y$, as desired. 
\end{proof}
\begin{lemma}\label{lemma 1}
    Let $A$ be a monadic quasi-implication algebra. Then: 
     \[R^M_A[\{x\}]^{\perp^M_A}=\{y\in A\setminus 0:y\cdot (\Diamond x\cdot 0)=1\}.\]
\end{lemma}
\begin{proof}
    For the left-to-right inclusion, take any $y\in R^M_A[\{x\}]^{\perp}$. Now observe that $\Diamond x\cdot\Diamond x=1$ for all $x\in A$ and hence $\Diamond x\in R^M_A[\{x\}]$ so $y\perp^M_A\Diamond x$. Thus we have $y\cdot(\Diamond x\cdot 0)=1$ by the definition of $\perp^M_A$. 

    For the right-to-left inclusion, suppose $y\cdot(\Diamond x\cdot 0)=1$ for some $y\in A\setminus\{0\}$. It suffices to show that $y\perp^M_Az$ for all $z\in R^M_A[\{x\}]$, i.e., that $y\preceq z^*$ for all $z\in R^M_A[\{x\}]$. Therefore let $z\in R^M_A[\{x\}]$. Since $y\cdot(\Diamond x\cdot 0)=1$, we have $y\preceq(\Diamond x)^*$. Moreover, since $z\in R^M_A[\{x\}]$, we have $z\preceq\Diamond x$. Since $^*$ is an orthocomplementation, it is order-inverting and hence $(\Diamond x)^*\preceq z^*$ and thus $y\preceq z^*$ by the transitivity of $\preceq$, as desired.    
\end{proof}
\par
\vspace{1cm}
\begin{lemma}\label{lemma 2}
     If $A$ is a monadic quasi-implication algebra, then $z\cdot\Diamond(\Diamond x\cdot 0)=1$ for all $z\in R^M_A[R^M_A[\{x\}]^{\perp^M_A}]$. 
\end{lemma}
\begin{proof}
    Take any $z\in R^M_A[R^M_A[\{x\}]^{\perp^M_A}]$ so that $yR^M_A z$ for some $y\in R^M_A[\{x\}]^{\perp^M_A}$. Then $y\cdot(\Diamond x\cdot 0)=1$ as well as $z\cdot\Diamond y=1$, i.e., $y\preceq(\Diamond x)^*$ and $z\preceq\Diamond y$. By Proposition \ref{monotone}, $\Diamond$ is monotone so the former inequality yields $\Diamond y\preceq\Diamond(\Diamond x)^*$ and hence $z\preceq\Diamond(\Diamond x)^*$ by the transitivity of $\preceq$, so $z\cdot\Diamond(\Diamond x\cdot 0)=1$.  
\end{proof}

\begin{theorem}\label{monadic maclaren frame}
    If $A$ is a monadic quasi-implication algebra, then its monadic MacLaren frame $X^M_A$ forms a monadic orthoframe. 
\end{theorem}
\begin{proof}
 Theorems \ref{thm 4.5} and \ref{thm 4.6} allow us to generalize the proof of Proposition 7.27 in \cite{harding} to the case of monadic quasi-implication algebras. By condition 2(a) of Definition \ref{def of mqia}, we have $x\cdot\Diamond x=1$ and hence $xR^M_Ax$ for all $x\in A\setminus\{0\}$ so $R^M_A$ is reflexive. To see that $R^M_A$ is transitive, assume that $xR^M_Ay$ and $yR^M_Az$ so that $y\cdot\Diamond x=1$ and $z\cdot\Diamond y=1$. Therefore $y\preceq\Diamond x$ and $z\preceq\Diamond y$. Applying Proposition \ref{monotone} to the fact that $y\preceq\Diamond x$ yields $\Diamond y\preceq\Diamond\Diamond x$. Then by Proposition \ref{idempotent}, we have $\Diamond\Diamond x=\Diamond x$ and hence $\Diamond y\preceq\Diamond x$. Thus by the transitivity of $\preceq$, we have $z\preceq\Diamond x$ so $z\cdot\Diamond x=1$ and hence $xR^M_Az$, as desired.

    We now verify that $R^M_A[\{x\}]^{\perp^M_A}$ is closed under $R^M_A$, i.e., $R^M_A[R^M_A[\{x\}]^{\perp^M_A}]\subseteq R^M_A[\{x\}]^{\perp^M_A}$ for each $x\in A\setminus\{0\}$. First note that we have: \[R^M_A[\{x\}]=\{y\in A\setminus \{0\}:y\cdot \Diamond x=1\}\] and by Lemma \ref{lemma 1}, we have: 
    \[R^M_A[\{x\}]^{\perp^M_A}=\{y\in A\setminus \{0\}:y\cdot (\Diamond x\cdot 0)=1\}.\] Hence choose any $y\in R^M_A[\{x\}]^{\perp^M_A}$ with $yR^M_Az$, i.e., $z\in R^M_A[R^M_A[\{x\}]^{\perp^M_A}]$. Then: \[y\cdot(\Diamond x\cdot 0)=z\cdot\Diamond y=1.\] Therefore $z\cdot\Diamond(\Diamond x\cdot 0)=1$ by Lemma \ref{lemma 2}, but $\Diamond(\Diamond x\cdot 0)=\Diamond x\cdot 0$ so $z\cdot(\Diamond x\cdot 0)=1$ which by Lemma \ref{lemma 1} implies that $z\in R^M_A[\{x\}]^{\perp^M_A}$. 
\end{proof}

For a monadic ortholattice $A$, Harding, McDonald, and Peinado \cite{harding3} constructed a monadic orthoframe from $A$ by defining a binary relation $R^G_A\subseteq \mathfrak{F}(A)\times \mathfrak{F}(A)$ on the Goldblatt frame $X^G_A=\langle \mathfrak{F}(A),\perp^G_A\rangle$ by $\alpha R^G_A\beta\Leftrightarrow \exists[\alpha]\subseteq\beta$ where $\exists[\alpha]=\{\exists x:x\in\alpha\}$. They went on to demonstrate that this construction gives rise to the canonical completion as well as the canonical extension of $A$. The following definition generalizes the above construction of a monadic orthoframe to the setting of monadic quasi-implication algebras. 

\begin{definition}
    Let $A$ be a monadic quasi-implication algebra. Then define a binary relation $R^G_A\subseteq\mathfrak{F}(A)\times\mathfrak{F}(A)$ by $\alpha R^G_A\beta\Leftrightarrow\Diamond[\alpha]\subseteq\beta$. 
\end{definition}

We call $X^G_A=\langle\mathfrak{F}(A);\perp^G_A,R^G_A\rangle$ the \emph{monadic Goldblatt frame} of $A$. 

\begin{theorem}\label{monadic qia is a gold}
    If $A$ is a monadic quasi-implication algebra, then its monadic Goldblatt frame $X^G_A$ forms a monadic orthoframe. 
\end{theorem}
\begin{proof}
Similarly to that of Theorem \ref{monadic maclaren frame}, Theorems \ref{thm 4.5} and \ref{thm 4.6} allow us to generalize the proof of Proposition 3.5 in \cite{harding3} to the case of monadic quasi-implication algebras. For reflexivity, let $\alpha\in\mathfrak{F}(A)$ and let $x\in\Diamond[\alpha]$ so that $x=\Diamond y$ for some $y\in\alpha$. Since $y\cdot\Diamond y=1$ and $\alpha$ is upward closed, we have $\Diamond y\in\alpha$ but $\Diamond y=x$ so $x\in\alpha$. Therefore $\Diamond[\alpha]\subseteq\alpha$ and hence $\alpha R^G_A\alpha$ for every $\alpha\in\mathfrak{F}(A)$.
    
    To see that $R^G_A$ is transitive, assume that $\alpha R^G_A\beta$ and $\beta R^G_A\gamma$ for arbitrary $\alpha,\beta,\gamma\in\mathfrak{F}(A)$ so that $\Diamond[\alpha]\subseteq\beta$ and $\Diamond[\beta]\subseteq\gamma$. Now suppose $x\in\Diamond[\alpha]$ so that $x=\Diamond y$ for some $y\in\alpha$. By hypothesis we have $\Diamond y\in\beta$ and hence $\Diamond\Diamond y\in\Diamond[\beta]$. Again, by hypothesis, we have $\Diamond\Diamond y\in\gamma$. Then by Proposition \ref{idempotent}, we have $\Diamond\Diamond y=\Diamond y$ and hence $\Diamond y=x\in\gamma$. Therefore $\Diamond[\alpha]\subseteq\gamma$ so that $\alpha R^G_A\gamma$, as desired. 

    We now verify that $R^G_A[R^G_A[\{\alpha\}]^{\perp^G_A}]\subseteq R^G_A[\{\alpha\}]^{\perp^G_A}$ for each $\alpha\in\mathfrak{F}(A)$. Let $\gamma$ be the smallest filter generated by $\Diamond[\alpha]$. Then $R^G_A[\{\alpha\}]^{\perp^G_A}=\{\gamma\}^{\perp^G_A}$ since $\gamma$ is the smallest filter belonging to $R^G_A[\{\alpha\}]$ by our hypothesis. Now observe that for any $\beta\in\{\gamma\}^{\perp^G_A}$, i.e., $\beta\perp^G_A\gamma$, there exists $y\in A$ such that $y\in \gamma$ and $y\cdot0\in\beta$. Therefore one can find $x_1,\dots,x_n\in\alpha$ satisfying: \[(\Diamond x_1\odot\cdots\odot\Diamond x_n)\cdot y=1\] since $\gamma$ is the filter generated by $\Diamond[\alpha]$. By fixing $x=x_1\odot\dots\odot x_n$, we have: \[\Diamond x\cdot(\Diamond x_1\odot\dots\odot\Diamond x_n)=1\] and hence $\Diamond x\cdot y=1$ since we already have $(\Diamond x_1\odot\cdots\odot\Diamond x_n)\cdot y=1$. Therefore $(y\cdot 0)\cdot(\Diamond x\cdot 0)=1$ and hence $(\Diamond x\cdot 0)\in\beta$ since $y\cdot0\in\beta$ and $\beta$ is a filter. Therefore, we find that: 
    \[R^G_A[\{\alpha\}]^{\perp^G_A}=\{\beta\in\mathfrak{F}(A):\Diamond x\cdot0\in\beta\hspace{.2cm}\text{for some $x\in\alpha$}\}.\] Finally, from the definition of $R^G_A$ and the fact $\Diamond(\Diamond x\cdot 0)=\Diamond x\cdot 0$ for any $x\in A$ by condition 2(b) of Definition \ref{def of mqia}, it follows that for $\beta,\omega\in\mathfrak{F}(A)$ such that $\beta\in R^G_A[\{\alpha\}]^{\perp^G_A}$ with $\beta R^G_A\omega$ i.e., $\omega\in R^G_A[R^G_A[\{\alpha\}]^{\perp^G_A}]$, we have $\omega \in R^G_A[\{\alpha\}]^{\perp^G_A}$, which establishes the desired inclusion.             
\end{proof}
\section{Conclusions and related lines of research}
In section 3, we generalized the constructions of an orthoframe obtained by MacLaren \cite{maclaren} and by Goldblatt \cite{goldblatt} to the setting of bounded quasi-implication algebras. In section 4, we introduced monadic quasi-implication algebras and demonstrated that the category $\mathbf{QMA}$ of quantum monadic algebras is isomorphic to the category $\mathbf{MQIA}$ of monadic quasi-implication algebras. In section 5, we generalized the constructions of a monadic orthoframe given by Harding \cite{harding} and by Harding, McDonald, and Peinado \cite{harding3} to the setting of monadic quasi-implication algebras. 

We note that the various constructions of a monadic orthoframe from a monadic quasi-implication algebra $A$ given in Theorem \ref{monadic maclaren frame} and Theorem \ref{monadic qia is a gold} do not exploit the orthomodularity of the lattice structure induced by $A$. Future lines of research include developing Kripke frames for quantum monadic algebras (i.e., Kripke frames whose induced complete lattices of bi-orthogonally closed subsets form quantum monadic algebras) and then generalizing such frames to the setting of monadic quasi-implication algebras. This will however lead to some technical challenges due to the fact that the condition of orthomodularity on an ortholattice is not first-order definable by an orthogonality relation alone. This is a well-known result due to Goldblatt \cite{goldblatt}.    
 
\section*{Declarations}
\subsection*{Ethical approval}
Not applicable. 
\subsection*{Funding}
This research was funded by the CGS-D SSHRC grant no.~767-2022-1514. 
\subsection*{Availability of data and materials}
Not applicable.

\end{document}